\documentclass[11pt]{amsart}
\usepackage{latexsym}
\usepackage{amsfonts}
\usepackage{hyperref}
\usepackage{xcolor}

\usepackage[top=1in, bottom=1.2in, left=1in, right=1in]{geometry}

\newtheorem{theorem}{Theorem}
\newtheorem{corollary}{Corollary}

\newtheorem{lemma}{Lemma}

\newtheorem{proposition}{Proposition}
\newtheorem{remark}{Remark}
\newcommand{\Z}{\mathbb{Z}}

\newcommand{\Q}{\mathbb{Q}}

\begin{document}

\author[]{Alexander Olshanskii}
\address{Department of Mathematics, Vanderbilt University, Nashville, TN 37240}
\email{alexander.olshanskiy@vanderbilt.edu}
\thanks{Research of the first author was partially supported by the NSF award DMS-1901976.} 

\author[]{Vladimir Shpilrain}
\address{Department of Mathematics, The City College of New York, New York,
NY 10031} \email{shpilrain@yahoo.com}

\title[Average-case complexity]{Linear average-case complexity of algorithmic problems\\
 in groups}

\begin{abstract}
The worst-case complexity of group-theoretic algorithms has been studied for a long time. Generic-case complexity, or complexity on random inputs, was introduced and studied relatively recently. In this paper, we address the average-case time complexity of the word problem in several classes of groups and show that it is often the case that the average-case complexity is linear with respect to the length of an input word. The classes of groups that we consider include groups of matrices over rationals (in particular, polycyclic groups), some classes of solvable groups, as well as free products. Along the way, we improve several bounds for the worst-case complexity of the word problem in groups of matrices, in particular in nilpotent groups.
For free products, we also address the average-case complexity of the subgroup membership problem and show that it is often linear, too. Finally, we discuss complexity of the {\it identity problem} that has not been considered before.

\end{abstract}

\maketitle

\section{Introduction}

The worst-case time complexity of group-theoretic algorithms has been studied probably since small cancellation groups were introduced and it was noticed that the word problem in these groups admits a linear time (with respect to the length of an input) solution, see e.g. \cite{LS} for relevant algorithms.

Genericity of group-theoretic properties and generic-case complexity of group-theoretic algorithms were introduced more recently, see \cite{KMSS}. The idea of the average-case complexity appeared in \cite{Knuth}, formalized in \cite{Levin}, and was addressed in the context of group theory for the first time in \cite{KMSS2}. (Problem 6 in our Section \ref{problems} touches upon relations between the average-case and generic-case complexity of the word problem.)  Specifically, the authors of \cite{KMSS2} addressed the average-case complexity of the word and subgroup membership problems in some non-amenable groups and showed that this complexity was linear. The fact that the considered groups were non-amenable played an essential role in computing the average-case complexity due to the fact that the cogrowth rate of any non-amenable group is strictly smaller than that of the ambient free group, see \cite{Cohen}, \cite{Grigorchuk}. 
In this paper, we consider classes of groups not considered in \cite{KMSS2} and use a different approach.

Studying average-case (as well as generic-case) complexity of an algorithm requires introducing complexity of inputs. We note that it is arguable which is more natural to consider as inputs of algorithms in group theory -- all group words in a given alphabet $X$ or just freely reduced words. It depends on how the inputs for one algorithm or another are ``sampled". Basically, both ways are natural and the average-case complexity typically does not depend on whether inputs are freely reduced or not, but considering all group words often makes counting arguments easier, and this is what we do in the present paper.

There are several ways to define the average-case time complexity of a group-theoretic algorithm $\mathcal{A}$ that takes words as inputs.
Let $W_n$ denote the set of all words of length $n$ in a finite group alphabet. For a word $w \in W_n$, let $T(w)$ denote the time that the algorithm $\mathcal{A}$ (for a multitape Turing machine) works on input $w$.

One can then define the average-case time complexity of the algorithm $\mathcal{A}$ on inputs of length $n$ as
\vskip -0.5cm

\begin{equation}\label{avcase1}
\frac{1}{|W_n|} \sum_{w\in W_n} T(w).
\end{equation}

In this paper, we address the average-case time complexity of the word problem and the subgroup membership problem in several classes of groups and show that it is often the case that the average-case complexity is linear with respect to the length of an input word, which is as good as it gets if one considers groups given by generators and defining relations, see \cite{sublinear}.


We also mention that the problem in some ``philosophical" sense dual to finding the average-case complexity of the word problem in a given group $G$ is finding ``non-cooperative" sets of words $w$ for which deciding whether or not $w=1$ in $G$ is as computationally hard as it gets; these words are ``responsible" for a high worst-case complexity of the word problem in $G$. This is relevant to properties of the {\it Dehn function} of $G$, see e.g. \cite{Birget}, \cite{MU} and references therein.

In the concluding Section \ref{identity}, we consider the {\it identity problem} whose complexity  has not been considered before, to the best of our knowledge. The problem is: given a variety $\mathcal{G}$ of groups and a group word $w=w(x_1, \ldots, x_m)$, find out whether or not $w(g_1, \ldots, g_m)=1$ for every group $G \in \mathcal{G}$ and every $g_1, \ldots, g_m \in G$. In the beginning of Section \ref{identity}, we explain the differences between the identity problem and the word problem.

To summarize, we establish the following results on the worst-case and the average-case time complexity:

\begin{itemize}

\item[1.] In finitely generated groups of matrices over rationals (in particular, in polycyclic
groups) the worst-case time complexity of the word problem is $O(n \log^2 n)$. For finitely generated nilpotent groups, we show that the worst-case time complexity of the word problem is, in fact, $O(n \cdot \log^{(k)} n)$ for any integer $k\ge 1$,  where $\log^{(k)} n$ denotes the function $\log \ldots \log n$, with $k$ logarithms. This improves the best previously known estimate of $O(n \log^2 n)$ in this case (cf. \cite{MMNV}).

The average-case time complexity of the word problem in groups of matrices is often linear; in particular, this is the case in all polycyclic groups. These results apply, in particular, to finitely generated nilpotent groups. See Section \ref{linear}.
%
\medskip

\item[2.] If a finitely generated group $G$ has polynomial-time worst-case complexity of the
word problem and maps onto 
a (generalized) lamplighter group, then the average-case time complexity of the word problem in $G$ is linear. This applies, in particular, to free solvable groups. See Section \ref{factor}.
\medskip

\item[3.]  In some popular groups including Thompson's group
$F$, the average-case time complexity of the word problem is linear, whereas the worst-case complexity is often superlinear.
See Section \ref{average matrices} and Section \ref{Thompson}.
\medskip

\item[4.] If a finitely generated group $G$ is a free product of nontrivial groups $A$ and $B$,
both having polynomial-time worst-case complexity of the word problem, then $G$ has linear average-case time complexity of the word problem. See Section \ref{freeproducts}.
\medskip

\item[5.] If a finitely generated group $G$ is a free product of nontrivial  groups $A$ and $B$,
both having polynomial-time worst-case time complexity of the subgroup membership problem, then $G$ has linear average-case time complexity of the subgroup membership problem. All subgroups in question should be finitely generated. See Section \ref{membership}.
\medskip

\item[6.] In any product of nilpotent varieties of groups, the worst-case time complexity of the identity
problem is polynomial. See Section \ref{identity}.
\medskip

\item[7.] If in a variety $\mathcal{M}$ of groups the worst-case time complexity of the identity
problem is $O(\exp(n^\theta))$, where $\theta < \frac{1}{2}$, then the average-case time complexity of the identity problem in $\mathcal{M}$ is linear. See Section \ref{identity}.

%

\end{itemize}

\subsection{The average-case complexity vs. the worst-case complexity} \label{Vegas}

There are some obvious, as well as less obvious, relations between the worst-case, the generic-case, and the average-case complexity of an algorithm. Some of them were discussed in \cite{KMSS2}, some are discussed in our Section \ref{problems}.

One obvious relation is that the average-case complexity of an algorithmic problem cannot be higher  than the worst-case complexity. With the word problem, it is often the case that the former is strictly lower than the latter. On the intuitive level, the reason is that inputs $w\in W_n$ for which $T(w)$ is high are ``sparse".
If one is able to quantify this ``sparsity", then one can try to split $W_n$ in a disjoint union $W_n=\cup_j W_n^{(j)}$ of sets $W_n^{(j)}$ where $T(w)$ is the same for any $w \in W_n^{(j)}$, so that the formula (\ref{avcase1}) is stratified as

\begin{equation}\label{avcase2}
\frac{1}{|W_n|} \sum_j |W_n^{(j)}| \cdot T(w \in W_n^{(j)}) = \sum_j \frac{|W_n^{(j)}|}{|W_n|} \cdot T(w \in W_n^{(j)}).
\end{equation}

This formula is not very practical though because the number of summands may be too large.
Keeping in mind that our goal is typically reduced to finding an {\it upper bound} for the average-case complexity rather  than its precise value, we will use a slightly different stratification that is more practical. Let $W_n=\cup_j V_n^{(j)}$ be a disjoint union.
Denote by $\overline{T}(V_n^{(j)})$ an upper bound on $T(w)$ for $w \in V_n^{(j)}$ and use the following sum instead of (\ref{avcase2}):

\begin{equation}\label{avcase3}
\sum_j \frac{|V_n^{(j)}|}{|W_n|} \cdot \overline{T}(V_n^{(j)}).
\end{equation}

In this sum, the number of summands can be rather small (in this paper, it is typically 2 or 3, although it can be larger, as in the proof of Theorem \ref{Theorem 7}),
depending on an upper bound on the average-case complexity one would like to establish.

Thus, our strategy will be to find suitable upper bounds on $\frac{|V_n^{(j)}|}{|W_n|}$ for those $V_n^{(j)}$ where $\overline{T}(V_n^{(j)})$ is high.

An alternative strategy (used in \cite{KMSS2}) is, for a given input, to run two algorithms in parallel. One algorithm, call it {\it honest}, always terminates in finite time and gives a correct result. The other algorithm, a {\it Las Vegas algorithm}, is a fast randomized algorithm that never gives an incorrect result; that is, it either  produces the correct result or informs about the failure to obtain any result. (In contrast, a {\it Monte Carlo algorithm} is a randomized algorithm whose output may be incorrect with some (typically small) probability.)

A Las Vegas algorithm can improve the time complexity of an honest, ``hard-working", algorithms that always gives a correct answer but is slow. Specifically, by running a fast Las Vegas algorithm and a slow honest algorithm in parallel, one often gets another honest algorithm whose average-case complexity is somewhere in between because there is a large enough proportion
of inputs on which a fast Las Vegas algorithm will terminate with
the correct answer to dominate the average-case complexity. This idea was used in \cite{KMSS2} where it was shown, in particular, that if a group $G$ has the word problem solvable in subexponential time and if $G$ has a non-amenable factor group where the word problem is solvable in a complexity class $\mathcal{C}$, then there is an honest algorithm that solves the word problem in $G$ with average-case complexity in $\mathcal{C}$.

\section{The word problem in linear groups}\label{linear}

Many interesting groups can be represented by matrices over a field. This includes polycyclic (in particular, nilpotent) groups \cite{Swan}, braid groups, and some mapping class groups. (In general it is unknown whether all mapping class groups are linear.) We note though that even where it is known that groups can be represented by matrices over a field of characteristic 0 (e.g. torsion-free metabelian groups \cite{Wehrfritz} or braid groups \cite{Bigelow}, \cite{Krammer}), it is often not known whether these groups can be represented by matrices over rationals.

Suppose $g_1, \ldots, g_r$ are generators of a group $G$ of matrices over integers, and let $w=w(x_1, \ldots, x_r)$ be a group word of length $n$. The algorithm that decides whether or not $w(g_1, \ldots, g_r)=1$ in $G$ is straightforward: going left to right, one multiplies by a matrix corresponding to $g_i^{\pm 1}$ that occurs next in the word $w$. This matrix multiplication takes time $O(N)$, where $N$ is the sum of lengths (in whatever presentation, say binary or decimal) of all entries of the current matrix. This $N$ is increasing (by an additive constant) after each multiplication by a $g_i^{\pm 1}$, so we have to consider $N$ to be $O(n)$. Since the total number of matrix multiplications is $n$, the algorithm takes time $O(n^2)$. In Section \ref{integers}, we show that this can actually be improved to $O(n \log^2 n)$, and not only for groups of matrices over integers, but also over rationals and over any finite algebraic extension of the field $\Q$ of rationals.

We start our discussion with a rather special class of groups of {\it unitriangular} matrices. These include all finitely generated torsion-free nilpotent groups.

\subsection{The word problem in nilpotent groups and in groups of unitriangular matrices}\label{nilpotent}

Finitely generated torsion-free nilpotent groups are representable by  unitriangular matrices over integers. When one multiplies uni(upper)triangular matrices, the entries do not grow exponentially with respect to the number of matrices (from a finite collection) multiplied. In fact, we have:

\begin{lemma}\label{unitriangular} Let ${\bf B}=\{B_1, \ldots, B_s\}$ be a finite collection of uni(upper)triangular matrices.
The absolute value of the $(i,j)$th entry in any product of $n$ matrices, each of which comes from this collection, is bounded by a polynomial in $n$ of degree $j-i$, and therefore the total bit length of the entries of a matrix exhibits sublinear growth $O(\log n)$.
\end{lemma}

\begin{proof}
Let $C \ge 2$ be a constant such that the absolute value of any entry $b_{ij}$ of any matrix $B_k$ is bounded by $C^{j-i}$.

Denote an arbitrary product of $n$ matrices, each of which comes from the collection ${\bf B}$, by $A(n)$.
We are going to show that for any entry $a_{ij}(n)$ of the matrix $A$, one has $|a_{ij}(n)|  \le C^{j-i} n^{j-i}$. We will use simultaneous induction on $(j-i)$ (with an obvious basis $j-i=0$) and on $n$ (with an obvious basis $n=1$).

Multiply $A(n-1)$ by one of the matrices $B_k$ on the right. Then, by using the inductive assumptions we get:

\noindent $|a_{ij}(n)|= |1\cdot b_{ij}+ a_{i, i+1}(n-1)b_{i+1,j}+...+
a_{i,j-1}(n-1)b_{j-1,j} + a_{i,j}(n-1)\cdot 1 |\\
 \le C^{j-1}(1 +(n-1)+...+(n-1)^{j-i})\le C^{j-i} n^{j-i}$ ~since $n \ge 2$.

\end{proof}

Therefore, we have:

\begin{lemma}\label{torsion-free1}
In any group $G$ of uni(upper)triangular matrices over integers generated by matrices $g_1, \ldots, g_k$, there is an algorithm $\mathcal A$ that evaluates any product of $n$ matrices $g_i^{\pm 1}$ in time  $O(n \cdot \log n)$.
\end{lemma}


%
%
%
%

The complexity can actually be pushed even further down. Denote by  $\log^{(k)} n$ the function $\log \ldots \log n$, with $k$ logarithms. To avoid situations where the function $\log^{(k)} n$ is undefined, we define $\log^{(k)} n =1$ if there is $j\le k$ such that $\log^{(j)} n <1$. Also, to be more specific, we are going to interpret the $\log$ function as logarithm with base 2.

\begin{proposition}\label{multilog}
For any integer $k \ge 1$, there is an algorithm $\mathcal A_k$ that evaluates any product $w$ of $n$ uni(upper)triangular matrices $g_i^{\pm 1}$ (from a finite collection) in time  $O(n \cdot \log^{(k)} n)$.

\end{proposition}

\begin{proof}
The proof is by induction on $k$, with the basis $k=1$ provided by Lemma \ref{torsion-free1}. 

The algorithm $\mathcal A_{k+1}$ for $k \ge 1$ works as follows. Note that we can compute an approximation of $\log n$, and therefore of $\log^k n$, by an integer $n_k$ up to a factor $(1+o(1))$, in time $O(n)$ since $\log n$ is the length of $n$ (in binary or whatever form).

Then, again in time $O(n)$, the input word $w$ of length $n$ can be split as a product of $m=\lfloor \frac{n}{n_k}\rfloor$ subwords $f_i$ of length $O(\log^k n)$, i.e., $w = f_1 \ldots f_m$.

Now the algorithm $\mathcal A_{k+1}$ works in $m$ stages. Suppose that after $(i-1)$ stages the algorithm has computed the matrix $A_{i-1}$ that corresponds to the product $f_1 \ldots f_{i-1}$.
At the stage $i$, the algorithm first computes the matrix $M_{i}$ that corresponds to the subword  $f_i$. This can be done by using the algorithm $\mathcal A_{k}$, and therefore this takes time $O(\log^k n \cdot \log^{(k)}(\log^k n))=O(\log^k n \cdot \log^{(k+1)} n)$, independent of $i$.

Continuing with the stage $i$, the algorithm $\mathcal A_{k+1}$ computes the product $A_{i-1} M_{i}$. Given the bound  on the size of entries in these matrices (see the first paragraph of this subsection), this computation takes time
$O(\log n \cdot (\log (\log^{k} n)))= O(\log n \cdot \log^{(2)}n)$. (Here we use the fact that multiplying a $k$-bit integer by an $l$-bit integer takes time $O(k\cdot l)$ if one uses the standard ``school algorithm" for integer multiplication.)

Thus, the whole stage $i$ takes time
 $O(\log^k n \cdot \log^{(k+1)} n) + O(\log n \cdot \log^{(2)}n)$.

Finally, we multiply the latter estimate by the number $O(\frac{n}{\log^k n})$ of stages and get\\
$(O(\log^k n \cdot \log^{(k+1)} n) + O(\log n \cdot \log^{(2)}n)) \cdot O(\frac{n}{\log^k n}) =  O(n \cdot \log^{(k+1)} n)$. This completes the proof.

\end{proof}

\begin{remark}\label{allows}
The inductive argument in the proof of Proposition \ref{multilog} allows one to compute constants $c_k >0$ such that $O(n \cdot \log^{(k)} n)$ in the statement of Proposition \ref{multilog} can be replaced by $\le c_k n \log^{(k)} n$. Moreover, the procedure of  constructing the algorithm $\mathcal A_{k+1}$ from $\mathcal A_{k}$ also allows one to compute a constant $c$ such that $c_{k+1}\le c \cdot c_k$ for any $k$.
\end{remark}

This can be improved  even further:

\begin{theorem} \label{sub}
There is an algorithm $\mathcal B$ that evaluates any product $w$ of $n$ uni(upper)triangular matrices $g_i^{\pm 1}$ (from a finite collection) in time at most $T(n)$, where $T(n)=o(n \log^{(k)} n))$ for any integer $k \ge 1$.

\end{theorem}

\begin{proof}
The plan of the proof is as follows. Given an input word $w$, the algorithm $\mathcal B$ will first determine the length $n$ of $w$; this can be done in time $O(n)$. Then $\mathcal B$ will compute an integer $k=k(n)$; this $k$ will tell which algorithm $\mathcal A_k$ (from Proposition \ref{multilog}) should be used on the word $w$.

Having figured out the right $k$, the algorithm $\mathcal B$ will create a program for the algorithm $\mathcal A_k$ (this step is needed because we cannot just store programs for infinitely many $\mathcal A_k$). This, too, will be done in time $O(n)$.

Now we are going to describe how exactly the integer $k=k(n)$ is computed. By Remark  \ref{allows}, one can recover (integer) constants $c_k>0$ such that the algorithm $\mathcal A_k$ works on any word $w$ of length $n$ no longer than $c_k \cdot n \log^{(k)} n$, and there is an (integer) constant $c>0$ such that  $c_k < c \cdot c_{k-1}$, implying $c_k < c^k$.

Denote by $\exp^{(k)} 2$ the tower of exponents $2^{2^{\ldots}}$ of height $k$. Then, given $n\ge 4$, let $k$ be the largest natural number such that $\exp^{(k)}2\le n$, i.e.,

\begin{equation}\label{exp}
2 \le \log^{(k)}n <4.
\end{equation}

Obviously, this $k=k(n)$ goes to infinity when $n$ goes to infinity. Also, this $k$ can be computed in time bounded by $c \cdot n$ for some constant $c>0$ because $\exp^{(k-1)} 2$ is not greater than the number of bits in the binary form of $n$.

Next, we note that the time of creating a program for the algorithm $\mathcal A_k$ is bounded by $c \cdot k < c \cdot n$ (we can use the same constant $c>0$ as in the previous paragraph for convenience) because a program for $\mathcal A_k$ can be obtained from a program for $\mathcal A_{k-1}$ by adding a few commands whose number is bounded by a constant.

Thus, the total runtime of the algorithm $\mathcal B$ is bounded by

\begin{equation}\label{exp2}
2cn +c^k n \log^{(k)} n \le  2c n + 8c^k n \le 10c^k n.
\end{equation}

\noindent (Recall that $\log^{(k)}n <4$ by (\ref{exp}).)

Now note that one can find an integer $k_0$ such that for any $k\ge k_0$ one would have\\
$10 c^k < \exp {(\lfloor k/2 \rfloor)} 2$, which is not greater than $\log^{(\lfloor k/2 \rfloor)} n$ by (\ref{exp}).
Therefore, if we select this $k_0$ and take $n$ to be large enough so that $k=k(n) > k_0$, then the right-hand side of (\ref{exp2}) will be bounded (from above) by $n \log^{(\lfloor k/2 \rfloor)} n$, i.e., the runtime of the algorithm $\mathcal B$ is asymptotically smaller than that of the algorithm $\mathcal A_{\lfloor k/2-1 \rfloor}$. Since $k$ goes to infinity when $n$ goes to infinity, the desired asymptotic estimate is obtained.

\end{proof}

\begin{remark}
It may be interesting to note that our proof of Theorem \ref{sub} uses what can be considered a more general version of the classical divide-and-conquer method and the Master theorem \cite{Master} because in our argument, the number of subwords into which an input is split {\rm depends on the length of the input}.
\end{remark}

Theorem \ref{sub} implies that there is an algorithm that solves the word problem in finitely generated torsion-free nilpotent groups and has the worst-case time complexity $O(n \cdot \log^{(k)} n)$ for any $k \ge 1$. An easy extra argument allows one to extend this observation to {\it all} finitely generated nilpotent groups:

\begin{theorem} \label{nilpotent}
For any finitely generated nilpotent group $G$, there is an algorithm with the worst-case time complexity $O(n \cdot \log^{(k)} n)$ for any integer $k \ge 1$ that solves the word problem in $G$.
\end{theorem}

\begin{proof}
Let $G$ be a finitely generated nilpotent group. By a well-known result of \cite{Hirsh}, $G$ can be embedded in a direct product of a (finitely generated) torsion-free nilpotent group and a finite group. Therefore, the complexity of the word problem in $G$ is the same as it is in a finitely generated torsion-free nilpotent group, i.e.,  is $O(n \cdot \log^{(k)} n)$ for any $k \ge 1$ by Theorem \ref{sub}.
\end{proof}

In \cite{MMNV}, the worst-case complexity of the word problem in finitely generated nilpotent groups was shown to be $O(n \cdot \log^2 n)$.

Proposition \ref{Theorem 1} in the next section will imply that the average-case complexity of the word problem in any finitely generated nilpotent group is linear.

\subsection{The word problem in groups of matrices over integers or rationals}\label{integers}

Groups of matrices over integers are a special class of linear groups. In this class, the worst-case complexity of the word problem is quasilinear, as we show in this section. This class includes, in particular, polycyclic groups (see e.g. \cite{Wehrfritz}).

\begin{proposition}\label{integers}
In a finitely generated group (or a semigroup) of matrices over integers (in particular, in any polycyclic group),  the worst-case time complexity of the word problem is $O(n \log^2 n)$.
\end{proposition}

\begin{proof}
The key ingredient of the proof is a reference to a fast method of integer multiplication colloquially known as ``Karatsuba multiplication". This direction of research was opened by  Karatsuba \cite{Karatsuba}, and the best known complexity result to date is due to Harvey and van der Hoeven \cite{Harvey} who used a much more sophisticated technique than Karatsuba's divide-and-conquer method. The main result of \cite{Harvey} says that there is an algorithm (for a multitape Turing machine) for multiplying two $n$-bit integers, based on discrete Fourier transforms, of time complexity at most $C \cdot n \log n$, where $C$ is some positive constant. Since addition of two $n$-bit integers has complexity $O(n)$, this result
implies that multiplication of matrices over $\Z$ whose entries have bit length bounded by $m$ has complexity bounded by $C_1 \cdot m \log m$ for some $C_1>0$.

Now our algorithm for the word problem is as follows. Let the input word $w=w(g_1, \ldots, g_r)$ have length $n$. Here $g_1, \ldots, g_r$ are generators of a given group of matrices. For convenience of the exposition, we will assume that $n=2^k$ for some $k>0$. (If not, we can complement $w$ by a product of the identity matrices.)

We are going to write our $n$ matrices (factors of $w$) on the input tape of a Turing machine. This takes time $O(n)$, and the length of the whole transcript is $N=O(n)$ as well. Then we are going to use a
``divide-and-conquer" kind of algorithm that will work in stages, cutting the number of matrices in our transcript in half at every stage by replacing pairs of neighboring matrices with their product.

At the stage $i$, we will have $2^{k-i+1}$ factors remaining. If the bit length of all the matrix entries at the previous stage $(i-1)$ was bounded by $C_0 \cdot (2^i-1)$ for some $C_0>0$, then after multiplying pairs of neighboring matrices at stage $i$ we will get an upper bound\\
$C_0\cdot (2^i-1) + C_0\cdot (2^i-1) +C_0 =  C_0(2^{i+1}-1)\le C_0 n$. The extra $C_0$ summand is due to the fact that addition of numbers can increase the bit length but by not more than $C_0$.
Since the number of matrices gets cut in half at every stage of the algorithm, it follows that the total bit length of all entries of all matrices is still bounded by $N=O(n)$.

Note that the number of stages of our algorithm is $k = \log n$. Therefore, the total runtime of this algorithm is bounded by
$\log n \cdot (C_1 (C_0 \cdot n \log (C_0 n) + O(n))) \le C_2 \cdot n \log^2 n$ for some $C_2>0$.
\end{proof}

The proof of Proposition \ref{integers} actually goes through for groups of matrices over rationals as well because if multiplication of integers takes time $O(n \log n)$, then so does multiplication of rationals, only in the latter case $n$ is the bound on the length of numerators and denominators of all entries in a matrix. Thus, we have:

\begin{proposition}\label{rationals}
In a finitely generated group (or a semigroup) of matrices over rationals, the worst-case complexity of the word problem is $O(n \log^2 n)$.
\end{proposition}

%
%

We note also that Proposition \ref{rationals} can be further generalized to matrices over algebraic extensions of the field $\Q$ of rationals. Indeed, any finite algebraic extension of $\Q$ can be embedded in the algebra of $k \times k$ matrices over $\Q$, for some $k$. Therefore, any $r \times r$ matrix over a finite extension of $\Q$ can be replaced by a $rk \times rk$ matrix over $\Q$. Since $k$ does not depend on $n$ (the length of an input word), the result follows.

\begin{remark}
While the main result of \cite{Harvey} gives the best known upper bound $C \cdot n\log n$ for the complexity of $n$-bit integers multiplication, we note that a much earlier result of \cite{Strassen} gave the  $O(n\log n\log\log n)$ estimate. Using this latter result would slightly worsen the worst-case complexity estimate of our Proposition \ref{integers}, but  would not change our (linear) estimate of the average-case complexity in Theorem \ref{Theorem 1}.

\end{remark}

\subsection{The average-case complexity of the word problem in polycyclic groups}\label{average matrices}

We start with a (slightly re-phrased and simplified) technical result from  \cite{Woess} that we are going to use in this and subsequent sections. In the lemma below, 
by the probability of the event ``$w = 1$ in $G$" we mean the probability for a random walk  corresponding to the word $w$, on the Cayley graph of $G$, to return to 1.

\begin{lemma}\cite[Theorem 15.8a]{Woess} \label{lemma Woess}
Let $G$ be a polycyclic group, and suppose that all words of length $n$ (in the given generators of $G$) are sampled with equal probability.
Then one has the following alternative:
\medskip

\noindent {\bf (a)} $G$ has polynomial growth with degree $d$ and for a random word $w$ of length $n$, the probability of the event $w = 1$ is $O(\frac{1}{n^{d/2}})$.

\medskip

\noindent {\bf (b)} $G$ has exponential growth and  the probability of the event $w = 1$
is $O(\exp(-c \cdot n^{\frac{1}{3}}))$ for some constant $c>0$.

\end{lemma}

We are going to use part (a) of this alternative in the present paper. Recall that by a well-known result of Gromov \cite{Gromov}, a group has polynomial growth if and only if it is virtually nilpotent. We will be mostly using the ``if" part of this result, which was previously known, see  \cite{Bass}.

Now we are ready for one of our main results on the average-case complexity:

\begin{theorem}\label{Theorem 1}
In any polycyclic group, the average-case complexity of the word problem is linear.

\end{theorem}

\begin{proof}
Let $G$ be an infinite polycyclic group. It has a factor group $H=G/N$ that is virtually free abelian of rank $r \ge 1$.
Thus, we first check if an input word $w$ is equal to 1 modulo $N$; this takes linear time in $n$, the length of $w$. Indeed, let $w=g_1 \cdots g_{n}$ and let $K$ be an abelian normal subgroup of $H$ such that $H/K$ is finite. Then first check if $w=1$ in $H/K$; this can be done in linear time. If $w \in K$, then rewrite $w$ in generators of $K$ using the Reidemeister-Schreier procedure. This, too, takes linear time. Indeed, if we already have a Schreier representative for a product  $g_1 \cdots g_{i-1}$, then we multiply its image in the factor group $H/K$ by the image of $g_i$ and get a representative for $g_1 \cdots g_{i}$. This step is performed in time bounded by a constant since the group $H/K$ is finite. Therefore, the whole rewriting takes linear time in $n$.

The obtained word $w'$ should have length $\le n$ in generators of $K$. Then we solve the word problem for $w'$ in the abelian group $K$, which can be done in linear time in the length of $w'$, and therefore the whole solution takes linear time in $n$.

If $w=1$ modulo $N$, then we apply an algorithm for solving the word problem in $N$ with the worst-case complexity $O(n \log^2 n)$; such an algorithm exists by Proposition \ref{integers}. (We remind the reader again that polycyclic groups are representable by matrices over integers \cite{Wehrfritz}.) The probability that $w=1$ modulo $N$ is  $O(\frac{1}{n^{1/2}})$ by Lemma \ref{lemma Woess}.

Therefore, the expected runtime of the combined algorithm is $O(n) + O(\frac{1}{n^{1/2}}) \cdot O(n \log^2 n)  = O(n)$. This completes the proof.

\end{proof}

\begin{remark}
The same proof actually goes through in a more general situation, for finitely generated groups of matrices over $\Q$ that have a virtually abelian factor group. Thus, in particular, the average-case time complexity of the word problem in any virtually solvable linear group over $\Q$ is linear. This includes Baumslag-Solitar groups $BS(1,n)$. (They are not polycyclic.)

\end{remark}

\section{The word problem in groups with a solvable factor group}\label{factor}

In \cite{KMSS2}, the average-case complexity of the word problem was shown to be linear for a class of groups that have a non-amenable factor group. In this section, we consider a class of groups with a ``dual" property, namely groups that have a solvable (and therefore amenable) factor group.

\subsection{The word problem in groups that map onto a (generalized) lamplighter group}\label{solvable}

In this section, we discuss the average-case complexity of the word problem for a special class of groups, namely for groups that map onto a (generalized) lamplighter group, i.e., the wreath product of $\Z_p$ and $\Z$, for a prime $p$.
Of course, among such groups, even among solvable ones, there are groups with unsolvable word problem \cite{OK}, so for those groups we cannot talk about the average-case complexity of the word problem.

Suppose a group $G$ has the word problem solvable in time  $O(n^k)$ for some $k \ge 1$ and suppose $G$ maps onto the wreath product of $\Z_p$ and $\Z$ that we denote by $W$.



Let $\varphi$ be a homomorphism  from  $G$ onto $W$. In $W$, the word problem is solvable in linear time, see e.g. \cite{Sale}.
Then, the density of the set of those  words in the group alphabet on $m$ letters that are equal to 1 in $W$ is $O(\exp(-c \cdot n^{\frac{1}{3}}))$ for some constant $c>0$, see \cite[Theorem 15.15]{Woess}. Therefore, the same estimate is valid for the density of the set of words in the kernel of the homomorphism $\varphi$.

Thus, we run the following two algorithms in succession. The first algorithm checks if $w \in Ker ~\varphi$; this takes time $O(n)$. If $w \in Ker ~\varphi$, then we run the second, honest, algorithm that solves the word problem in $G$ in time polynomial in $n$. Then the expected runtime of the combined algorithm is $$O(n) +  O(\exp(-c \cdot n^{\frac{1}{3}}))  \cdot O(n^k) = O(n).$$

Therefore, we have:

\begin{theorem}\label{Theorem 2}
If a finitely generated group $G$ has polynomial-time worst-case complexity of the word problem and maps onto the wreath product of $\Z_p$ and $\Z$, then the average-case time complexity of the word problem in $G$ is linear.
\end{theorem}

Since the worst-case complexity of the word problem in free solvable groups is $O(n^2 \log n)$ by \cite{Ushakov}, we have

\begin{corollary}
The average-case time complexity of the word problem in any finitely generated free solvable group is linear.
\end{corollary}

Using a slightly different argument, the latter corollary can be generalized in a different direction. We call a variety $\mathcal S$ of groups {\it solvable} if all groups in $\mathcal S$ are solvable.

\begin{theorem}\label{relatively free}
Let $G$ be a finitely generated group that is free in some solvable variety $\mathcal S$. If $G$ has polynomial-time worst-case complexity of the word problem, then the average-case time complexity of the word problem in $G$ is linear.
\end{theorem}

\begin{proof}
It is known (see e.g. \cite{Groves}) that a solvable variety $\mathcal S$ satisfies the following alternative: either all groups in $\mathcal S$ are nilpotent-by-finite or $\mathcal S$ contains a (generalized) lamplighter group, i.e., the wreath product of $\Z_p$ and $\Z$, for a prime $p$.

If all groups in $\mathcal S$ are nilpotent-by-finite, then $G$ is polycyclic and therefore the average-case time complexity of the word problem in $G$ is linear by Theorem \ref{Theorem 1}.

Alternatively, suppose $\mathcal S$ contains the wreath product of $\Z_p$ and $\Z$ that we denote by $W$. 

 Let $G$ be a free group of the variety $\mathcal S$, with $r \ge 2$ free generators. Then $G$ maps onto $W$, and we can use the argument as in the proof of Theorem \ref{Theorem 2} to show that the average-case time complexity of the word problem in $G$ is linear.

\end{proof}

We note that there are many example of varieties satisfying the conditions of Theorem \ref{relatively free}. However, there are also examples of solvable varieties where free groups have unsolvable word problem, see \cite{Kleiman}. We also note that other examples of varieties where free groups have the word problem solvable with linear average-case time complexity are given in our Section \ref{identity}.

%
%

\subsection{The word problem in Thompson's group $F$}\label{Thompson}

One of the popular groups these days, Thompson's group $F$ (see e.g. \cite{CFP}), although not solvable or linear, admits treatment by our methods as far as the average-case complexity of the word problem is concerned. We also note that it is unknown whether or not this group is amenable, and therefore it is unknown whether it has any non-amenable factor groups, so methods of \cite{KMSS2} are not applicable to studying the average-case complexity of the word problem in this group.

\begin{proposition}\label{ThompsonWP}
The average-case time complexity of the word problem in Thompson's group $F$ is linear.
\end{proposition}

\begin{proof}
By \cite{SU}, the worst-case complexity of the word problem in $F$ is $O(n \log n)$. Then, the group $F$ has a presentation with two generators where all relators are in the commutator subgroup (see e.g. \cite{CFP}), so the factor group $F/[F, F]$ is free abelian of rank 2. By Lemma \ref{lemma Woess}, the probability for a random word $w$ in the generators of $F$ to be in the commutator subgroup is
$O(\frac{1}{n})$. Therefore, we can combine two algorithms (as in the proof of our Theorem \ref{Theorem 1}), one of which determines, in linear time, whether or not $w \in [F, F]$, and if $w \in [F, F]$, then the other algorithm determines in time $O(n \log n)$ whether or not $w=1$.

The expected runtime of this combined algorithm is  $O(n) + O(\frac{1}{n}) \cdot O(n \log n)  = O(n)$. This completes the proof.
\end{proof}

\section{The word problem in free products of groups}\label{freeproducts}

Here we prove the following

\begin{theorem}\label{free product}
Let $A$ and $B$ be nontrivial finitely generated groups and suppose the word problem in both $A$ and $B$ has a polynomial time solution.
Let $G=A \ast B$ be the free product of $A$ and $B$. Then the average-case time complexity of the word problem in $G$ is linear.  In fact, the average-case time complexity of reducing an element of  $G$ to the standard normal form is linear.

\end{theorem}

First we make some preliminary observations. We our going to scrutinize a natural algorithm, call it $\mathcal A$, for solving the word problem in $G$, and show eventually that its average-case time complexity (i.e., the expected runtime) is linear with respect to the length of an input word.

We may assume that at least one of the groups $A$ and $B$ has more than 2 elements. Let the group $A$ be generated by nontrivial $a_1, \ldots, a_s$ and the group $B$ by nontrivial $b_1, \ldots, b_t$.

Any word $w$ in the alphabet $a_1^{\pm 1}, \ldots, a_s^{\pm 1}, b_1^{\pm 1}, \ldots, b_t^{\pm 1}$ can be reduced to either  the empty word or to a word of the form

\begin{equation}\label{form}
w_1 w_2 \cdots w_k,
\end{equation}

\noindent where for any $i \le k$, the word $w_i$ represents either a non-trivial element of $A$ or a non-trivial element of $B$, and two successive $w_i$ represent elements of different groups. We will call a {\t syllable} of $w$ any subword of $w$ that is a word just in the $A$-alphabet or just in the $B$-alphabet, and is maximal with respect to this property.


The corresponding simple algorithm, that we call $\mathcal A$, works in stages. The $(n+1)$th stage of this algorithm consists of several steps and targets the $(n+1)$th letter of the input word $w$.

Suppose that after $n$ stages of the algorithm $\mathcal A$, the prefix $v=w(n)$ of length $n$ of the word $w$ is already in the form $w_1 w_2 \cdots w_i$ as above. Let $x$ be the $(n+1)$th letter of the word $w$. Then the $(n+1)$th stage of the algorithm $\mathcal A$ proceeds as follows. If $x$ is not an element of the alphabet that $w_i$ is a word in, then $w_{i+1}=x$.

Now suppose $x$ is from the same alphabet as $w_i$ is. Then check (using  algorithm for the word problem in either $A$ or $B$) whether $w_i x=1$ in the relevant free factor. If $w_i x=1$, we get rid of $w_i$ in the prefix $v=w(n)$ and get  $w(n+1)$ in the form (\ref{form}). Otherwise, we replace $w_i$ by $w_i x$ in $v$.

If the form (\ref{form}) of the prefix $v=w(n)$ satisfies the conditions stated after the display (\ref{form}), then this form for the prefix $v'=w(n+1)$ of length $(n+1)$ also satisfies these conditions.

We say that a word is {\it simple} if (of type $A$ or $B$) if it is nontrivial and is equal in
$A \ast B$ to an element of $A$ or $B$ (respectively). For example, any one-letter word is simple.

Any word $w$ can be represented as a product of subwords $w=u_1u_2,\ldots$, as follows. If $u_1, u_2,\ldots, u_j$ are already constructed and $w$ is graphically equal to $w=u_1u_2,\ldots, u_j u$, where $u \ne 1$, then $u_{j+1}$ is the simple prefix of maximum length of the word $u$.

\begin{lemma}\label{prefix}
Suppose that for some $j$, the prefix $v$ of length $n$ of a word $w$ is $u_1 u_2\cdots u_j$ (where $u_i$ are defined in the previous paragraph), then in the free product $A \ast B$ one has $u_i=w_i$ for $i = 1, \ldots, j$, where $w_1 w_2\cdots w_j$ is the form (\ref{form}) of the word $v$. Moreover, no letter in $w_1 w_2\cdots w_j$ will be changed in the course of further applying the algorithm $\mathcal A$.
\end{lemma}

\begin{proof}
We use induction by $j$. The statement is trivial for $j=0$, so let $j \ge 1$ and assume that the statement holds for all $k<j$.

Then, upon getting $w_1 w_2\cdots w_{j-1}$, the algorithm $\mathcal A$ works only with the suffix $u$ of $w=u_1 u_2\cdots u_{j-1} u$. Therefore, eventually the algorithm $\mathcal A$ will get to the last letter of the prefix of $u_j$ and produce the form (\ref{form}) for $u_j$. Since $u_j$ is simple, this form should have just one syllable $w_j$.

Let us assume, without loss of generality, that $w_j$ is an element of $A$.
Subsequent steps of the algorithm $\mathcal A$ could change the word $w_j$ only if, upon removing some subwords equal to 1, immediately after $w_j$  we got a letter from the $A$-alphabet, i.e., in $u$ we would have a prefix of the form  $u_j u' a$, where $u'$ is equal to 1 in $A \ast B$. This prefix then would be a simple word in the $A$-alphabet, in contradiction with the choice of $u_j$ as the simple prefix of maximum length of the word $u$. This completes the proof of the lemma.

\end{proof}

Now let the time complexity of the word problem in $A$ and $B$ be $O(n^d)$ for some $d$, where $n$ is the length of an input word. Then:

\begin{lemma}\label{runtime}
The  worst-case complexity of the algorithm $\mathcal A$ is $O(n^{d+1})$.
\end{lemma}

\begin{proof}
When expanding a prefix of an input word $w$ by adding one letter, the algorithm $\mathcal A$ checks (in the worst case) whether the extended syllable is equal to 1 in the relevant free factor. This takes time $O(n^d)$. Multiplying this by the number of letters in $w$, we get $O(n^{d+1})$.

\end{proof}

Now we get to the average-case complexity of the algorithm $\mathcal A$.
Informally, the average-case complexity of $\mathcal A$ is linear because long simple factors of an input $w$ are rare. Thus, we are now going to estimate proportions of words with a substantial number of long factors. Before we proceed with technical lemmas, we recall, for the record, a fact that we will be using repeatedly:
the number of words of length $r$ in a group alphabet with $m$ letters is $(2m)^r$.

\begin{lemma}\label{number}
{\bf (a)} There is a positive constant $c_1$ such that more than half of all words of length $n>0$ are not equal in $G$ to any words of length $< c_1 n$;
\medskip

\noindent {\bf (b)} The number of words of length $r$ equal to 1 in the group $G= A \ast B$ is at most $(2m)^r \cdot \exp (-c_2 \cdot r)$, where $m=s+t$ is the number of generators of $G$ and $c_2$ is a positive constant independent of $r$.
\end{lemma}

\begin{proof}

Denote by $P(n)$ the set of words in a group alphabet on $m$ letters that are equal to 1 in $G$. The number $\limsup_{n\to\infty} P_n^{\frac{1}{n}}$ is called the  spectral radius of the simple random walk on the Caley graph of $G$, or just the spectral radius of $G$.

First we note that, since the group $G= A \ast B$ contains a non-cyclic free group, it is non-amenable. Therefore, by \cite[Corollary 12.5]{Woess}, the spectral radius of  $G$ is $<1$.  Then there is a positive constant $c_1$ such that more than half of all words of length $n$ are not equal in $G$ to any words of length $< c_1 n$ (see \cite[Proposition 8.2]{Woess}).

Also, it follows from the definition of the spectral radius that there is a positive constant $c_2$ such that the number of words of any length $r$ equal to 1 in $G$ is at most $(2m)^r \cdot \exp (-c_2 r)$, where $m=s+t$ is the number of generators of the group $G$ and $(2m)^r$ is the total number of words of length $r$ in the union of the $A$-alphabet and the $B$-alphabet.
\end{proof}

\begin{lemma}\label{estimate}
There is a positive constant $c$ such that the number of simple words $w$ of any length $k$ {\color{blue}of type $A$} is at most $(2m)^k \exp(-c \cdot p)$, where $p = \lfloor \sqrt{k} \rfloor$, the floor function of $\sqrt{k}$.

\end{lemma}

\begin{proof} Let $c_1$ and $c_2$ be constants defined in Lemma \ref{number}.
Assume first that $w$ has a subword $v$, of length $q \ge c_1 p$, which is equal to 1 in $G= A \ast B$. By Lemma \ref{number}, the number of words of length $q$ equal to 1 in $G$ is at most $(2m)^q \cdot \exp (-c_2 \cdot q)$.

Write the word $w$ as $w=v'vv''$. If the length $r$ of the prefix $v'$ is fixed, the number of words $w$ like this is at most $(2m)^r (2m)^q \exp(-c_2q)(2m)^{k-r-q} =
\exp(-c_2q) (2m)^k$. Since $r$ can take less than $k$ different values, we have less than $k \exp(-c_2 q)(2m)^k$ different words containing a simple subword of length $q$.


Summing up over all $q \ge c_1 p$, we get (as the sum of a geometric progression) a number of the magnitude $O((2m)^k (k+1) \cdot \exp (-c_2 c_1 p))$. Upon decreasing $c_i$ if necessary, we can get rid of $(k+1)$ and get the claimed estimate in this case.

Now we consider the case where $w$ is a simple word without simple subwords of length $q \ge c_1 p$ equal to 1 in $G$.

By the definition of $p$, it is possible to split the word $w$ into subwords $w_1, \ldots, w_l$ of fixed lengths $p_i \ge p$, where $l \ge p$.
It is sufficient to show that, for large enough $k$ and $p$, there are no more than $(2m)^{p_i} \lambda$ ways to select each $w_i$, for some constant $\lambda < 1$. Indeed, in that case the number of different words $w$ will be not greater than $(2m)^k \lambda^p$, as desired.

By Lemma \ref{number}, there are more than $\frac{1}{2}(2m)^{s_i}$ different words $v_i$ of length $s_i$, where $|s_i - \lfloor \frac{p_i}{2}\rfloor| \le 1$, such that the geodesic form $g_i$ of $v_i$ in $A \ast B$ has length $> c_1 s_i$.

Let $v_i$ and $v_i'$ be two words like that. If the last letter of the geodesic form $g_i$ of $v_i$ and the first letter of the geodesic form $g_i'$ of $v_i'$ are from different alphabets, then, if $g_i$ and $g_i'$ are of length $\ge c_1s_i$, the length of the geodesic word $g_i g_i'$ has length $> c_1s_i$ and has a letter $b$ from the $B$-alphabet at the concatenation of $g_i$ and $g_i'$.


A letter from the $B$-alphabet (we keep the same notation $b$ for it)  will still be there after any reduction of the
word $u_i=v_iv_i'$ to the standard form, and moreover the
letter $b$ will be at a distance $\ge c_1s_i - 1$ from
the beginning as well as from the end of $u_i$.

If the last letter of $g_i$ and the first letter of $g_i'$ are from the same alphabet  (say, from the $A$-alphabet), then, by inserting between them a letter $u$ from the $B$-alphabet, we will again get a $B$-letter surrounded by $A$-letters, and it will be at a distance $\ge c_1s_i - 1$ from
the beginning as well as from the end of $u_i$.

The number of words $u_i$ like that is greater than $c_3 (2m)^{s_i}$ for some positive constant $c_3$ because $v_i$ and $v_i'$ were selected from a set that contains more than half of the words of length $s_i$.

For any words $w_i$ obtained as above, this letter $b$ will not cancel out upon any reduction of $w=w_1  \cdots w_d$ to the standard form. Indeed, for $b$ to cancel out, there should be a subword of $w$ equal to 1 in $G$ and having length $\ge c_1 p$, but there are no subwords like that. Thus, we get a contradiction with the simplicity of type $A$ of the word $w$.

Therefore, in a factorization $w=w_1  \cdots w_l$, one can select each $w_i$ from a set of cardinality at most $(2m)^{s_i}-c_3(2m)^{s_i} =(1-c_3)(2m)^{s_i}$. Hence we can get at most $(2m)^k (1-c_3)^p$ of products of such $w_i$, which implies the desired estimate in the second case as well.

\end{proof}

Let now $M(k)$ be the set of all simple group words of length $k$ in the given alphabet with $m$ letters. By Lemma \ref{estimate}, the cardinality of $M(k)$ is at most $\exp(-c p)(2m)^k$ for some $c>0$, where $p = \lfloor \sqrt{k} \rfloor$.

Then, let $M(n, k, l)$ denote the set of group words of length $n$ in the given alphabet that have at least $l$ non-overlapping subwords of length $k$ from $M(k)$.

Let us now estimate the ratio of the number of words in $M(n, k, l)$ to the total number $(2m)^n$ of words of length $n$. Let us call this ratio the {\it density} of the set $M(n, k, l)$. Recall also that our assumption is that the time complexity of the word problem in the groups $A$ and $B$ is $O(n^d)$ for some $d$, where $n$ is the length of an input word. Then we have:

\begin{lemma}\label{density} There is a constant $C>0$ such that:\\
{\bf (a)} If $k>C$ and  $l = \lfloor \frac{n}{k^{d+3}} \rfloor$, then the density of the set $M(n, k, l)$ is at most
$\exp (-c \cdot n/(k^{d +3}))$, where the constant $c$ is provided by Lemma \ref{estimate}.
\medskip

\noindent {\bf (b)} The density of the set $M(n, k, 1)$ is at most $n^{-d-1}$  if $k > C \cdot (\log n)^2$.

\end{lemma}

\begin{proof}
We start with part (a). Every word in $M(n, k, l)$ can be described as follows. Think of $n$ bins separated by $2l$ dividers (enumerated left to right), so that between any two dividers numbered
$2j-1$ and $2j$, there are $k$ bins where we can ``put" a word from $M(k)$. In the remaining bins, we just put arbitrary words.  We note that positions of all $2l$ dividers are  unambiguously determined by positions of $l$ odd-numbered dividers.

Between the dividers numbered $2j-1$ and $2j$, one can put no more than  $\exp(-c p)(2m)^{k}$ different words by Lemma \ref{estimate}. Summing up over all such pairs of dividers, we get at most $\exp(-cpl)(2m)^{kl}$ different possibilities for words in $M(n, k, l)$. Then, there are $(2m)^{n - kl}$ ways to write letters in the remaining places.

Thus, altogether we have at most $\exp(-cpl)(2m)^n$ words in $M(n, k, l)$ with a fixed position of all the dividers.  The density of the set of words from $M(n, k, l)$ with a fixed position of dividers is therefore at most
\begin{equation} \label{eq2}
\exp(-cpl)  < exp(-cn/2k^{d+2.5}).
\end{equation}

Now select $k\ge C$. Let us get an upper bound on the number of dividers. For an upper bound, it is sufficient to just count the number of ways to position $2l$ dividers without any conditions on the size of spaces between them. This latter number is equal to $\binom{n+2l}{2l}$.

By the Stirling formula, this is bounded from above by $\frac{(n+l)^{n+l}}{(l)^{l} n^n}$. After canceling $n^{n+l}$ in the numerator and the denominator, we get $(1+\varepsilon)^{n+l}$ in the numerator, where $\varepsilon = \frac{l}{n} \le \frac{1}{k^{d+3}}$. Thus, the
numerator is less than $\exp (2\varepsilon n)$. In the denominator, we have $\varepsilon^{\varepsilon n}$. Therefore, $\binom{n+2l}{2l} \le \exp (2 \epsilon n +|log \epsilon|(2n/k^{d+3}) < \exp (3 n/k^{d+3} + (d+3) log k (2n/k^{d+3}).$

Combining this with (\ref{eq2}), we get that the density of the set $M(n, k, l)$ is at most $$\exp (-cn/2k^{d+2.5}+3 n/k^{d+3} +(d+3) \log k(2n/k^{d+3})).$$ Summing up the exponents, we get an exponent bounded from above by $-cn/k^{d+3}$ because the constant $C$ can be selected so that if $k > C$, one has $c/2k^{2.5} > 3/k^{d+3}+(d+3)\log k/k^{d+3} + c/k^{d+3}$. This completes the proof of part (a) of the lemma.


Now, for part (b), let us estimate the number of words in the set $M(n, k, 1)$, with $k > C (\log n)^2$. Here we have an upper bound of $\exp(-cp)$ instead of $\exp(-cpl)$ in (\ref{eq2}), and
$\binom{n}{2}$  (instead of $\binom{n}{l}$) is bounded just by $n$. Thus, the number of words in $M(n, k, 1)$ is bounded from above by

$$n \cdot \exp(-cp) \le n \exp (-c \cdot C^{1/2}\log n) = n^{-c \cdot C^{1/2} + 2} \le n^{-d-1},$$

\noindent if $C$ is selected large enough compared to both $c$ and $d$. This completes the proof of part (b) of the lemma.

\end{proof}

\medskip

\noindent {\bf Proof of Theorem \ref{free product}.} It is clearly sufficient to prove the ``In fact" part of the theorem.  Let $w$ be an input word of length $n$ of the algorithm $\mathcal A$ (see remarks after the statement of Theorem \ref{free product} in the beginning of this section). Recall that by Lemma \ref{runtime}, the worst-case complexity of the algorithm $\mathcal A$ is bounded by $Kn^{d+1}$ for some constant $K$. The constants $c$ and $C$ below are provided by Lemma \ref{estimate} and Lemma \ref{density}, respectively.

Let (\ref{form}) be the canonical form of $w$.
To estimate the average-case complexity of the algorithm $\mathcal A$, let us split the set of all words of length $n$ as a union of the following subsets $V_n^{(j)}$, and then estimate contribution of each $V_n^{(j)}$ to the average-case complexity, cf. formula (\ref{avcase3}) in the Introduction.

\medskip

\noindent {\bf 1.} Let $V_n^{(1)}$ denote the set of words of length $n$ that contain a simple subword of length $\ge C(\log n)^2$. By Lemma \ref{density}(b), the density of this set is $O(n^{-d-1})$, and by Lemma \ref{runtime}, the runtime is $\overline{T}(V_n^{(1)}) \le K n^{d+1}$ for some constant $K$.
Therefore, contribution of the corresponding summand in formula  (\ref{avcase3}) will be $O(1)$.
\medskip

\noindent {\bf 2.}  Let $V_n^{(2)}$ denote the set of words $w$ of length $n$ that are not in $V_n^{(1)}$ and for any integer $k < C(\log n)^2$, $w$ contains less than $[n/k^{d+3}]$ simple subwords of length $k$.

Then by Lemma \ref{prefix}, for any integer $k$ in the interval $(C, C(\log n)^2)$, the reduced form of the word $w$ will have less than $[n/k^{d+3}]$   syllables of length $k$ produced by the algorithm $\mathcal A$ while working on subwords of length $k < C(\log n)^2$. Therefore, producing these syllables takes time $\le K(k^{d+1}[n/k^{d+3}]) <2K n/k^2$. Summing up over all relevant values of $k$, we see that producing all syllables of length $\ge C$ takes time $O(n)$.

Producing other syllables takes time $\le C_0$ for each, for some constant $C_0$, because those syllables are produced by the algorithm $\mathcal A$ while working on subwords of bounded length $\le C$. Thus, the total time spent on producing those syllables is  $\le C_0 n$, so the total contribution of the set $V_n^{(2)}$ to the  formula  (\ref{avcase3}) will be $O(n)$.
\medskip

\noindent {\bf 3.}  All the remaining words of length $n$ form a set that we denote by
$V_n^{(3)}$. This set is the union, over all $k$, $C<k< C(\log n)^2$, of the subsets $M(n,k,l)$, where  $l=\lfloor n/k^{d+3}\rfloor$.

By Lemma \ref{density}(a), density of any $M(n,k,l)$ is $\le \exp(-cn/k^{d+3})$.
Therefore, density of the whole set $V_n^{(3)}$ is less than $C (\log n)^2 \exp(-c n/k^{d+3})$. After multiplying this last quantity by the runtime $O(n^{d+1})$  of the algorithm $\mathcal A$ on the set $V_n^{(3)}$, we get that the contribution of the set $V_n^{(3)}$ to the  formula  (\ref{avcase3}) will be $O(1)$.

Summing up, we see that the average-case time complexity of the algorithm $\mathcal A$ is $O(n)$, and this completes the proof.  $\Box$



\section{The subgroup membership problem in free products of groups}\label{membership}

Let $G$ be a group and $H$ a finitely generated subgroup of $G$.
We say that the membership problem in $H$ is solvable in polynomial time if, given a word $w$ of length $n$ representing an element of $G$, one can decide in time polynomial in $n$ whether or not $w$ actually represents an element of $H$. Then we have:

\begin{theorem}\label{membershipproblem} Let $G=A \ast B$ be the free product of groups $A$ and $B$. Suppose that for any finitely generated subgroup $H$ of $A$ or $B$ the membership problem in $H$ is solvable in polynomial time $O(n^d)$. Then:

\medskip

\noindent {\bf (a)} For any finitely generated subgroup $H$ of $G$, the membership problem in $H$ is solvable in polynomial time $O(n^{d+1})$;
\medskip

\noindent {\bf (b)} Suppose a subgroup $H \subset G$ has trivial intersection with any subgroup of $G$ conjugate (in $G$) to $A$ or $B$. Then the (worst-case) time complexity of the subgroup membership problem in $H$ is not higher than the time complexity of putting words into free product normal form in $G$.
\medskip

\noindent {\bf (c)} If $A$ and $B$ are nontrivial, then the average-case complexity of the
subgroup membership problem for $G$ is linear.

\end{theorem}

First we make some preliminary remarks. We are going to use the subgroup graph of a given $H$, similar to the Stallings graph (see e.g. \cite{KWM}), as follows.

Every element $w$ in $G=A \ast B$ has a normal form

\begin{equation}\label{form2}
(g_1) h_1 \cdots g_r (h_r),
\end{equation}

\noindent where $g_i$ and $h_i$ represent nontrivial elements of $A$ and $B$, respectively, and the factors in parentheses may or may not be present. Based on this form, one can build a simple directed circuit with edges labeled by syllables of the normal form (\ref{form2}). Thus, we are going to have ``$A$-edges" and ``$B$-edges", and when traversing this circuit starting at a fixed vertex $V_0$, we are going to read the normal form of $w$. An edge labeled $s$ traversed in the opposite direction would have label $s^{-1}$.

Let us build circuits like this for every generator of a given subgroup $H$, with a common initial
vertex $V_0$. Denote the obtained graph by $\Gamma_1$. Thus, labels of circuits with the initial
vertex $V_0$ in $\Gamma_1$ represent elements of $H$, and every element of $H$ is the label of a
circuit like that.

Now suppose that the graph $\Gamma_1$ has two different vertices connected by a path $\pi$ whose label is an element equal to 1 in $G=A \ast B$. 
Then identify these two vertices. The obtained graph $\Gamma_2$ (with fewer vertices) represents elements of $H$ as well as $\Gamma_1$ does.

Indeed, suppose an element $g$ of $G$ is represented by a path $\mu$ in the graph $\Gamma_2$.
Then the same $g$ can be read off a combination of paths $\mu_1, \mu_2, \ldots$ in the graph $\Gamma_1$, where $\mu_1$ ends at one of the aforementioned vertices, $\mu_2$ starts at the other one, etc. Then in $\Gamma_2$, one can insert (several times) the path $\pi$ and get the path
$\mu_1 \pi^{\pm 1} \mu_2 \pi^{\pm 1} \ldots$ in $\Gamma_1$, and this path will represent the same element $g \in G$.

Continuing with gluing vertices that way, we will eventually get a graph $\Gamma'$ where any path connecting different vertices represents a nontrivial element of $G$.

In the case where the intersection of $H$ with any subgroup of $G$ conjugate to $A$ or $B$ is trivial, one can simplify the graph even further. Suppose there is a simple circuit $\pi = e_1 \ldots e_k, ~k \ge 1,$ in the graph $\Gamma'$, so that either all $e_i$ are $A$-edges or all $e_i$ are $B$-edges. Then the element equal to the label $Lab(\pi)$ should be trivial in $A$ since otherwise the element equal to $Lab(\mu)Lab(\pi)Lab(\mu)^{-1}$ would be a non-trivial element of $H$. Here $\mu$ is the path that connects the vertex $V_0$ with the initial vertex of the circuit  $\pi$. Since $\pi$ is simple, every edge in $\pi$, including $e_k^{\pm 1}$, occurs in $\pi$ just once. In particular, the edge $e_k$ is not among $e_1, \ldots, e_{k-1}$, and therefore $e_k$ can be removed from $\Gamma'$ and in any path that contained $e_k$, the path $e_{k-1}^{-1}\ldots e_1^{-1}$ should be used instead.

Thus, in this case one can get rid of all simple circuits, and therefore of all circuits, all of  whose edges are $A$-edges or all of whose edges are $B$-edges. Denote the obtained graph by $\overline{\Gamma}$.

\medskip

\noindent {\bf Proof of Theorem \ref{membershipproblem}.}
Let us call a maximal connected subgraph $K$ of $\overline{\Gamma}$ or $\Gamma'$ an $A$-component (or a $B$-component) if all edges of $K$ are $A$-edges (respectively, $B$-edges).

If in an $A$-component $K$ one fixes a vertex $v$, then the pair $(K, v)$ defines a subgroup
$A(v) \le A$ that consists of elements equal (in $A$) to labels of circuits (in $K$) with the initial vertex $v$. This subgroup is finitely generated because the graph $K$ is finite.

Summing up, an element of $G$ represented by a word $w$ belongs to the subgroup $H$ if and only if
$w$ is equal (in $G$) to the label of some circuit in $\Gamma'$, with $V_0$ as both the initial and terminal point.

Let $w = w_1\ldots w_k$ be the form (\ref{form}) of the word $w$, obtained by applying the algorithm
$\mathcal A$ from Section \ref{freeproducts}. Then, if $w$ is the label of some circuit $\mathcal C$ in $\overline{\Gamma}$, every syllable $w_i$ of $w$ should be read off either an $A$-subpath or a $B$-subpath of $\mathcal C$. This is because different vertices cannot be connected by a label representing the trivial element of $G=A \ast B$.

By way of induction on $k$ (the syllable length of $w$), suppose that, starting at the vertex $V_0$ and ending at a vertex $v$, we have already read (in the graph $\Gamma'$) an element that is equal to $w_1\ldots w_{k-1}$ in the group $G$.

Suppose that $w_k$ is an $A$-word. To read $w_k$ now, we need to find, in the $A$-component $K$ of the vertex $v$, another vertex $v'$ such that $v$ and $v'$ can be connected by a path $\pi$ (in $K$) whose label is equal to $w_k$ in the group $A$. We note, in passing, that if such a vertex $v'$ exists, then it is unique since otherwise, two such vertices $v'$ and $v''$ would be connected by
a path with the label $w_k^{-1}w_k = 1$.

Since the graph $K$ is connected, there is {\it some} path $\mu$ connecting $v$ and $v'$.
The path $\pi$ (that we are looking for), when combined with $\mu$, yields a circuit $\pi \mu^{-1}$. Therefore, the existence of $\pi$ that we are looking for is equivalent to the following condition:
the word $w_k u$, where $u$ is the label of the path $\mu^{-1}$, represents an element of the subgroup $A(v)$ of the group $A$. This takes us to the subgroup membership problem in the group $A$; this problem is solvable in time at most $f(|w_k|+c)$, where the function $f=O(n^d)$ bounds the time needed for deciding membership in the subgroup $A(v)$ (by one of the hypotheses of Theorem \ref{membershipproblem}), and the constant $c$ bounds the lengths of paths between $v$ and $v'$.

If the word $w_k u$ does not represent an element of $A(v)$, we move on to another vertex $v''$ in the graph $K$. To complete our proof by induction on $k$, note that the step from the word
$w_1\ldots w_{k-1}$ to the word $w_1\ldots w_{k-1} w_k$ is possible if and only if we eventually find a suitable vertex in $K$, so we need to check all vertices in the same $A$-component of the graph $K$. If not, then we already can say that $w$ does not belong to the subgroup $H$.

If we do get to $w_k$, we still need the last vertex to be $V_0$. Since the vertex $v'$, as well as subsequent vertices (if any), are unique, the process of searching for paths $\pi$ does not branch out and thus does not lead to exponential complexity. Since we already know from the proof of Theorem \ref{free product} that reducing an input word $w$ to the form (\ref{form}) takes time
$O(n^{d+1})$, and then the above procedure takes time $O(n^{d})$, the whole algorithm for deciding membership in $H$ takes time $O(n^{d+1})$, and this completes the proof of part (a).

For part (b), note that all $A$-components and $B$-components of the graph $\overline{\Gamma}$ constructed above are trees. Therefore, the membership problem in $H$ in that case is reduced to the word problem in $A$ or $B$, and this completes the proof of part (b).

For part (c), recall that the average-case complexity of reducing an input word $w$ to the form (\ref{form}) is linear in $n=|w|$, by Theorem \ref{free product}. From the proof of Theorem \ref{free product}, one can use the quantitative estimates of the distribution of all
words $w$ between different sets $M(n, k, l)$. What is left now is to replace the upper bound on the time complexity of the word problem for a simple word $u_i$ of length $l$ (in the proof of Theorem \ref{free product}) by the time complexity of deciding the membership of a syllable $w_i$ of length $l$ in a subgroup $A(v)$ or $B(v)$. Since both estimates are $O(n^{d})$, the final result is the same: the average-case time complexity of the membership problem in the group $G=A \ast B$ is linear.

$\Box$

\section{The identity problem}\label{identity}

In this section, we consider an algorithmic problem that is, on the surface, very similar to the word problem, and yet it is different.

The {\it identity problem} is: given a variety $\mathcal{G}$ of groups and a group word $w=w(x_1, \ldots, x_m)$, find out whether or not $w(g_1, \ldots, g_m)=1$ for any group $G \in \mathcal{G}$ and any $g_1, \ldots, g_m \in G$.

This problem is equivalent to the word problem in the free group {\it of infinite rank} in the variety $\mathcal{G}$. Note that a word $w(x_1,\ldots , x_n)$ is an identity in the variety $\mathcal{G}$ if and only if it is equal to 1 in the free group $F_n=F(x_1,\ldots , x_n)$ of this variety, see \cite[13.21]{Neumann}. However, here we consider arbitrary words $w(x_1,\ldots , x_m)$ without a bound on the number $m$ of variables, and therefore we are trying to solve the word problem in the free group $F(x_1,\ldots )$ of infinite rank in the variety $\mathcal{G}$.

Therefore, results of our Section \ref{factor} are not immediately applicable here. Also, we note that while the worst-case time complexity of the word problem in a finite group $G$ is always linear (one can ``evaluate" a given word $w=w(x_1, \ldots, x_m)$ step by step using the Cayley multiplication table for $G$), this is not the case with the identity problem. Indeed, a straightforward algorithm for checking if $w$ is an identity on $G$ would be verifying that $w(g_1, \ldots, g_m)=1$ for any $m$-tuple $g_1, \ldots, g_m$ of elements of $G$. This procedure is exponential in $m$.

We also note that there are examples of varieties given by just a single identity where the identity problem is not algorithmically solvable, see \cite{Kleiman}.

Furthermore, there is one interesting nuance here: if the rank of a group is infinite, then there is no bound on the number of different generators that occur in the input word $w$. For example, if
$w$ has length $n$, then $w$ can be a product of $n$ different generators $x_1, \ldots, x_n$. Then the (increasing) indexes of the generators $x_i$ should be considered part of the input, and in particular, should be taken into account when assessing complexity of the input.

More specifically, a word $w$ of length $n$ has complexity $O(n\log n)$ because, say, the index $n$ of the generator $x_n$ requires storage of size $O(\log n)$. Therefore, when we say that the input word $w$ has complexity $n$, this actually means that $w$ has length $O(\frac{n}{\log n})$. We note though that ``polynomial time in $n$" is equivalent to ``polynomial time in $\frac{n}{\log n}$". With that in mind, we have:

\begin{theorem}\label{nilvariety} {\bf (a)} In any variety $\mathcal{N}$ of nilpotent groups the worst-case time complexity of the identity problem is polynomial.

\medskip

\noindent {\bf (b)} If in two varieties $\mathcal{B}$ and $\mathcal{C}$ the worst-case time complexity of the identity problem is polynomial, then in the product variety $\mathcal{BC}$ it is polynomial, too. In particular, in any polynilpotent variety the worst-case time complexity of the identity problem is polynomial. (Recall that the product variety $\mathcal{BC}$ consists of all groups that are extensions of a group from $\mathcal{B}$ by a group from $\mathcal{C}$.)
%

\end{theorem}

\begin{proof}
For part (a), let $w$ be an input word with $n$ letters. In particular, $w$ involves $\le n$ generators.
Thus, we want to find out whether or not $w$ is equal to 1 in the group $G_n=F_n(\mathcal{N})$, the free group of rank $n$ in the variety $\mathcal{N}$.

Let $c$ be the nilpotency class of the variety $\mathcal{N}$. By \cite[Corollary 35.12]{Neumann}, the group $G_n$ can be embedded in the direct product of
$\binom{n}{c}$ copies of the group $G_c$. In the proof of \cite[Corollary 35.12]{Neumann}, the embedding is constructed algorithmically, for each factor of the direct product. Therefore, one can explicitly determine all ``coordinates" of the word $w$ in this direct product. This can be done in time polynomial in the length of $w$; this follows from the procedure described in the proof of \cite[Theorem 35.11]{Neumann}.

Thus, we have to solve the word problem in the group $G_c$ $\binom{n}{c}$ times. Since the word problem in any nilpotent group is solvable in quasilinear time (see our Proposition \ref{nilpotent}), the whole procedure takes polynomial time in $n$, and therefore also polynomial time in $O(n\log n)$, the complexity of $w$.

For part (b), we make one remark first. The free group of the product $\mathcal{B} \mathcal{C}$ of two varieties can be obtained as follows. Let $F$ be a (absolutely) free group (i.e., free in the variety of all groups) and let $K$ be a verbal subgroup of $F$ such that the factor group $F$ is free in the variety $\mathcal{C}$. Since $K$ itself is a free group, there is a factor group $K/V(K)$ that is free in the variety $\mathcal{B}$, for an appropriate verbal subgroup $V(K)$. Then the group $F/V(K)$ is free in the variety $\mathcal{B} \mathcal{C}$ by \cite[21.12]{Neumann}. Therefore, it is sufficient for our purposes to address the following general goal: to get a polynomial upper bound for the time complexity of the identity problem in groups of the form $F/V(K)$ where $K$ is a normal subgroup of $F$ such that the word problem in the group $F/K$ can be solved in polynomial time.

Thus, let $F$ be a (absolutely) free group with generators $x_1, x_2,\ldots$, and $K$ a normal subgroup of $F$ such that the word problem in the group $F/K$ is solvable in polynomial time in the length of an input word $w$. Thus, if $w \notin K$, this can be detected in polynomial time.

Now suppose that $w \in K$. We claim that for any such word $w=w(x_1,\ldots, x_m)$ of length $n \ge m$ one can rewrite $w$ in terms of free generators of $K$, in polynomial time in $n$, and the obtained word will have length $\le n$ in these generators.

Indeed, going left to right along the word $w$, one can apply Schreier's rewriting process, while at the same time computing Schreier representatives for cosets in $F/K$. Visually, Schreier representatives are selected by using a maximal tree in the Cayley graph of the group $F/K$. This tree does not have to be geodesic, i.e., the path connecting a vertex to the root in this tree  does not have to be the shortest connecting path in the whole Cayley graph. However, polynomial time complexity of the tree constructing process is due to the fact that part of the tree is being constructed simultaneously with applying Schreier's rewriting process to $w$, and it is sufficient to find just a finite subset of some free generating set of $K$ such that $w$ belongs to the subgroup generated by this finite subset.

In more detail, a step $\#(i+1)$ of the algorithm is as follows. Suppose that, going left to right along the word $w$, we already have prefixes $w_k$ with the following properties:
\medskip

\noindent {\bf 1.} For any prefix $w_k$ of length $k\le i <n$ of the word $w$, a Schreier representative $\bar w_k$ of length $\le k$ of its $K$-coset has already been selected.

\noindent {\bf 2.} For any two prefixes that are in the same $K$-coset, the same Schreier representative has been selected.

\noindent {\bf 3.} Any prefix of any selected Schreier representative has been selected, too.
\medskip

Let now $w_{i+1}$ be the prefix of length $(i+1)$ of the word $w$, i.e., $w_{i+1}=w_i a$ for some letter $a$. Then we check whether $w_{i+1}$ belongs to a $K$-coset of one of the already selected Schreier representatives. (This can be checked in polynomial time, by the hypotheses of part (b) of the theorem.) If it does, then (by definition) $\bar w_{i+1}$ is already on our list of
selected Schreier representatives. If it does not, then we let $\bar w_{i+1} = (\bar w_i) a$ and add this $\bar w_{i+1}$ to the list of selected Schreier representatives.



Now we make a list of free generators (and their inverses) of the group $K$, where we include all nontrivial words of the form $(\bar w_i) a (\bar w_{i+1})^{-1}$, for all prefixes $w_{i+1}$ of $w$ of length $(i+1)$, ~$i=0,1, \ldots, n-1$. Then $w = \prod_i (\bar w_i) a (\bar w_{i+1})^{-1}$ in $F$. This product can be rewritten in terms of free generators $y_1, y_2, \ldots$ of the group $K$ if one discards trivial factors from this product and denotes equal (or reciprocal) factors by the same letter $y_j^{\pm 1}$.

The length of the word $w'=w'(y_1,\ldots, y_p)$ obtained after rewriting $w$ this way in free generators $y_1, y_2, \ldots$ of the subgroup $K$ is still $\le n$, although the generators $y_1,\ldots, y_p$ that occur in $w'$ after rewriting can be different from the original $x_1,\ldots, x_m$ and the number of generators can increase.

To complete the proof of part (b), we take $K$ to be the verbal subgroup of $F$ that corresponds to the variety $\mathcal{C}$, i.e., the group $F/K$ is free in the variety $\mathcal{C}$, and therefore the identity problem in $F/K$ is solvable in polynomial time. Furthermore, the free group in the product variety $\mathcal{BC}$ is $F/V(K)$, where $V(K)$ is the verbal subgroup of $K$ that corresponds to the variety $\mathcal{B}$ (see e.g. \cite[Proposition 21.12]{Neumann}). In the free group $K/V(K)$ of the variety $\mathcal{B}$ the identity problem is solvable in polynomial time.


Thus, to find out if $w$ is an identity in $\mathcal{BC}$, we first check if $w$ is equal to 1 in $F/K$, the free group of the variety $\mathcal{C}$. If it is, then we rewrite $w$ in the generators of $K/V(K)$, the free group of the variety $\mathcal{B}$, by using the algorithm described above, and then check if $w=1$ in that free group. All steps can be done in time polynomial in the length of $w$. This completes the proof of part (b).


\end{proof}

Before we formulate our result on the average-case complexity of the identity problem, we need to somehow define a {\it  density} (this is what we denote by $\frac{|V_n^{(j)}|}{|W_n|}$ in formula (\ref{avcase3})) in the set of all (finite) group words in a countable alphabet $x_1, x_2, \ldots $. Below we describe a natural way to do that, where different subwords of the same length occur in long words with approximately equal probabilities. Alternative ways (where the latter property may not hold) were discussed in \cite{BMS}.

Every letter $x_i$ will be encoded in a 4-letter alphabet $\{x, y, 0, 1\}$ as follows. The letters $x_1$ and  $x_1^{-1}$ will be encoded as $x$ and $y$, respectively. These are 2 words of length 1. Then, the letters $x_2^{\pm 1}$ and  $x_3^{\pm 1}$ will be encoded as $x0, y0$ and $x1, y1$, respectively. These are 4 words of length 2. Then, for $x_4^{\pm 1}$ through $x_7^{\pm 1}$ we have 8 words of length 3: $x00, y00, \ldots, x11, y11$, and so on.

Then each of the $\frac{1}{2}4^n$ possible words of length $n$ in the alphabet  $\{x,y,0,1\}$  that start with $x$ or $y$ encodes a unique word of length $\le n$ in the original countable alphabet $X =\{x_1, x_2, \ldots \}$. Denote by $W_n$ the set of all such words of length $n$ in the alphabet $X$ (i.e., words that admit an encoding of length $\le n$ in the alphabet $\{x, y, 0, 1\}$) and say that each word in $W_n$ has complexity $n$. By the {\it density} of a set $V$ of words in the alphabet $X$ we will mean the ratio $\frac{|V_n|}{|W_n|} = 2^{1-2n}\cdot |V_n|$, where $V_n=V \cap W_n$. 
We note that $|V_n| = |U_n|$, where $U_n$ is the set of words of length $n$ in the 4-letter alphabet above that consists of encodings of the words in $V_n$.

This definition of density allows us to use formula (\ref{avcase3}) to get a linear upper bound on the average-case time complexity of the identity problem in some varieties:

\begin{theorem}\label{Theorem 7}
Suppose in a variety $\mathcal{M}$ of groups the worst-case time complexity of the identity problem is $O(\exp(n^\theta))$, where $\theta < \frac{1}{2}$. Then the average-case time complexity of the identity problem in $\mathcal{M}$ is linear.
\end{theorem}

We start with

\begin{lemma}\label{lemma7}
Let $V$ be the set of words (in a countable alphabet) of complexity $k$ that do not contain a particular letter $x_i$ or $x_i^{-1}$. Then, for any $n \ge k, n\ge 4$, one has  $|V_n| = |U_n| < (4-\frac{4^{-k}}{k})^n.$
\end{lemma}

\begin{proof}
Suppose $x_i$ or $x_i^{-1}$ is encoded by a word $v$ in our 4-letter alphabet. Now given an arbitrary word $w \in U_n$, factor it as a product $w_1 \cdots w_s w_{s+1}$, where  $s =\lfloor  \frac{n}{k+1}\rfloor$, the words $w_1, \ldots, w_s$ are of length $k+1$, and let the length of $w_{s+1}$ be $l \le k$.

None of the words $w_1, \ldots, w_s$ can be equal to $vx$ or $vy$. Therefore, the number of possibilities for each $w_i, 1 \le i \le s$, is at most $4^{k+1}-2$, hence the number of possibilities for $w$ is at most $(4^{k+1}-2)^s 4^l \le 4^n \cdot (1 - 2/4^{k+1})^s$.

For the factor on the right we have:

$$(1 - 2/4^{k+1})^s<(1 - 2/4^{k+1})^{n/2k} =
((1 - 2/4^{k+1})^{1/2k})^n <(1- 1/(k \cdot 4^{k+1}))^n,$$

\noindent and this completes the proof of the lemma.  (For the latter inequality we use the fact that $(1-x)^n > 1-nx$ if $x\in (0,1)$ and the integer $n$ is greater than 1.)
\end{proof}

We now continue with the proof of Theorem \ref{Theorem 7}. In reference to Lemma \ref{lemma7}, let
$k \le c \log_4 n$ for some constant $c, ~0 < c < 0.5$. Then

$$\frac{4^{-k}}{k} > 4^{-c\log_4 n} \cdot (c\log_4 n)^{-1} > n^{-(c+1)/3}$$

\noindent for sufficiently large $n$. Therefore,

$$\frac{|V_n|}{|W_n|} < (4-\frac{4^{-k}}{k})^n/(4^n/2) <
2(1- n^{-(c+1)/3})^n < 2\exp(-n^{(2-c)/3}).$$

Now recall that there are $2^{k-1}$ different letter/inverse
letter pairs of complexity $k$. Therefore, with
$k \le c \log_4 n$, for the set $Z_n$ of all words $w$ of complexity $n$ such that there is at least one letter $x_i^{\pm 1}$ of complexity $k$  that does not occur in $w$,  we have

\begin{equation}\label{Zn}
|Z_n| <  2^{k-1} \exp(-n^{(2-c)/3})< 2^{c \log_4 n} \exp(-n^{(2-c)/3})
= n^{c/2} \exp(-n^{(2-c)/3}) < \exp(-n^c).
\end{equation}

\noindent for sufficiently large $n$. (The last inequality follows from the fact that $n^{c/2}\exp(n^c - n^{(2-c)/3})$ converges to 0 when $n \to \infty$  since $c<0.5.$)
Therefore, $\frac{|W_n - Z_n|}{|W_n|} > 1 - \exp(-n^c)$.

For the next lemma, we are going to introduce subsets $X_{n,j}$ of the set $W_n - Z_n$, for all $2^k$ letters of complexity $k$: ~$x_s, x_s^{-1}, x_{s+1},\ldots , x_t^{-1}$, where $s=2^{k-1}, t = 2^k-1$.

We note that the free group of rank $n$ of any nontrivial variety $\mathcal{X}$ admits a homomorphism on the direct product of $n$ cyclic groups of order $p$, for some $p \ge 2$, so that different free generators are mapped to generators of different direct factors, see \cite[13.52]{Neumann}. Therefore, for any identity $w(x_1, \ldots, x_n)=1$ of the variety $\mathcal{X}$ the exponent sum on every $x_i$ that occurs in $w$ should be divisible by $p$. We will first assume that $p \ge 3$ (this is possible unless the variety $\mathcal{X}$ satisfies the identity $x^2=1$). We therefore introduce the following auxiliary subsets, depending on $p$.

The subsets $X_{n,j}$ are defined as follows. $X_{n,0}$ is just $W_n - Z_n$. Then, $X_{n,1}$ consists of all words where a letter $x_s$ occurs as many times $\pmod p$ as the letter $x_s^{-1}$ does, i.e., the sum of the exponents on $x_s$ in such a word is 0 $\pmod  p$. If, in addition, we impose a similar condition on the letter $x_{s+1}$, then so defined words comprise $X_{n,2} \subset X_{n,1}$, etc. Finally, if we impose this condition on {\it all} letters of complexity $k$, we get a subset $X_{n,2^{k-1}}$ that we denote by $Y_n$.

We also define subsets $X'_{n,j}$ of the whole set $W_n$ in a similar way. The difference between $X'_{n,j}$ and $X_{n,j}$ therefore is that words in $X'_{n,j}$ do not have to contain all letters of complexity $k$.

\begin{lemma}\label{lemma9}
For $r = 1,2,\ldots , 2^{k-1}$, one has $|X_{n,r}| \le \frac{|W_n-Z_n|}{2^r}$. In particular,
$|Y_n| \le 2^{-2^{k-1}} \cdot |W_n-Z_n|$. Furthermore, $|X'_{n,r}| \le \frac{|W_n-Z_n|}{2^{r-1}}$ if $k\le c \log_4 n$.
\end{lemma}

\begin{proof}
First we note that to obtain the last inequality from the previous one, it is sufficient to show that $|Z_n| \le 2^{-2^{k-1}} \cdot |W_n - Z_n|$, which follows from (\ref{Zn}) if $k\le c \log_4 n$.

Now, keeping in mind that for any $r$, complexity of a letter $x_r$ is the same as that of $x_r^{-1}$, we note that upon replacing in a word $w \in X_{n,1}$ the leftmost occurrence of $x_s^{\pm 1}$ with $x_s^{\mp 1}$, we obtain a word $w' \in (W_n - Z_n)$ that no longer satisfies the condition on the exponents on $x_s$ summing up to 0. Moreover, by applying the same operation to $w'$, one can recover the original $w$. This implies the inequality $|X_{n,1}|\le \frac{|W_n-Z_n|}{2}$.

Then, with a similar operation applied to the letter $x_{s+1}$, a word from $X_{n,2}$ can be transformed to a word from $X_{n,1} - X_{n,2}$. Therefore, $|X_{n,2}|\le \frac{1}{2} |X_{n,1}|$, and the obvious induction completes the proof of the lemma.

\end{proof}

Let now $k=\lfloor c\log_4 n \rfloor$. By Lemma \ref{lemma9}, $\frac{|Y_n|}{|W_n|} \le 2^{-n^c/2}$.
Therefore, for a sufficiently large $n$, the density (in $W_n$) of words of complexity $n$ with the exponent sum on each letter equal to 0 $\pmod p$ is less than $\exp(-n^c) + 2^{-n^c/2}$. This expression is less than $\exp (-n^{\theta})$ if $\theta < c$ and $n$ is large enough.

To complete the proof of the theorem, we need to find a suitable upper bound on the average time it takes to check whether a given word $w$ of complexity $n$ has the exponent sum on each letter of complexity $k$ equal to 0 $\pmod p$. The latter condition means that $w$ belongs to all sets  $X'_{n,1}\supset \ldots \supset X'_{n,2^{k-1}}$. Thus, we start the verification process with the letter $x_s$. We need one more lemma:

\begin{lemma}\label{lemma10}
For any $i$, there is an algorithm that determines the number of occurrences of the letter $x_i$ of complexity $k \le n$ in a given word $w$ of complexity $n$, such that the time complexity of this algorithm is bounded by $C \cdot n$, where a constant $C$ does not depend on $i$.
\end{lemma}

\begin{proof}
An algorithm can be described as follows. We will use a Turing machine with two tapes. A letter $x_i$ of complexity $k$ is written on one of the tapes as a word of length $k$, call it $u_i$. The word $w$ of complexity $n$ is written on the other tape. Now we move the two heads of our Turing machine along the two tapes, letter by letter in sync. If the corresponding letters on the two tapes are not the same, then we move the head of the first tape to the beginning of the word $u_i$, keep the other head where it is, and resume the process starting with this configuration.

When we do encounter the letter $x_i$ on the second tape, we add 1 to the counter. At the end (when we get to the end of the second tape), we just read the result off the counter.

This algorithm takes linear time in $n$ because the number of returns of the first head to the beginning of the first tape is bounded by $n$.

\end{proof}

Thus, we can check in time $O(n)$ whether or not the word $w$ belongs to $X'_{n,1}$. If it does not, then we are done. If it does, then we move on to checking the exponent sum on $x_{s+1}$ to see if $w$ belongs to $X'_{n,2}$, and so on. Since by Lemma \ref{lemma9}, cardinalities of the sets in this chain (and therefore
the densities $\frac{|X'_{n,j}|}{|W_n|}$) decrease in geometric progression, by formula (\ref{avcase3}) the average time of our checking procedure will be $O(n) (1 + 1/2 +1/4+\ldots) = O(n)$.

Finally, we need to consider words where the exponent sum on every $x_i$ is divisible by $p$. The density of the set of these words is less than $\exp(-n^{\theta})$, as we mentioned before the statement of Lemma \ref{lemma10}. Therefore, the contribution of this set to the sum  (\ref{avcase3}) is also $O(n)$. This is because the equality of any word from this set to 1 is verified in time  $O(\exp (n^{\theta}))$ by one of the conditions of Theorem \ref{Theorem 7}, hence the formula (\ref{avcase3}) gives a linear estimate. This completes the proof of the theorem in the case $p \ge 3$.


The remaining case is where the variety $\mathcal{M}$ is of exponent $p=2$, i.e., satisfies the identity $x^2=1$. In this case, $\mathcal{M}$ is abelian, and for a word $w$ to be an identity on $\mathcal{M}$ it is {\it necessary and sufficient} that the sum of exponents on each letter in $w$ is even.

The difference from the case $p>2$ in estimating the average-case complexity is as follows. In the proof of Lemma \ref{lemma9}, to obtain inequalities like $|X_{n,1}|\le \frac{|W_n-Z_n|}{2}$, we replaced the leftmost occurrence of a letter $x_s^{\pm 1}$ with $x_s^{\mp 1}$, thereby changing the exponent sum $\pmod p$ on $x_s$. This trick does not work if $p=2$. Instead, one can split the set of all letters of complexity $k$ into pairs of letters with indices different by 1 and replace $x_i^{\pm 1}$ with a letter that was paired  with $x_i^{\pm 1}$. The number of sets $X_{n,r}$ in this case will be half of what it was in the case $p>2$, and the time needed to compute the exponent sum $\pmod p$ on two letters (instead of just one) will be doubled. Nonetheless, the linear time estimate $O(n) (1 + 1/2 +1/4+\ldots) = O(n)$
will still hold, and this completes the proof of Theorem \ref{Theorem 7}.

$\Box$

\section{Open problems}\label{problems}

In this section, we suggest some directions for future research. Some of the problems below are somewhat  informal while others are more precise.

\begin{itemize}

%

\item[1.]  Is it true that the average-case time complexity of the word problem in any
finitely generated metabelian group is linear?

We note that in \cite{KMS}, the authors have shown that there are finitely generated solvable groups of class 3 (even residually finite ones) with super-exponential worst-case complexity of the word problem. The average-case time complexity of the word problem in such groups cannot be polynomial since the super-exponential runtime, even on a very small set of words, will dominate everything else.

Also, unlike in solvable groups of derived length $\ge 3$ \cite{OK}, the word problem in any finitely presented metabelian group is solvable. This follows from the fact that these groups are residually finite, but the first explicit algorithm was offered in \cite{Timoshenko}.
\medskip

\item[2.] Is the average-case time complexity of the word problem in any one-relator group linear?
Polynomial?

We note that it is not known whether the worst-case time complexity of the word problem in any one-relator group is polynomial (cf. \cite[Problem (OR3)]{BMS}).
\medskip

\item[3.] {\bf (a)} Can our $O(n \log^2 n)$ estimate for the worst-case complexity of the
word problem in groups of matrices over integers (or rationals) be improved?
\medskip

\noindent {\bf (b)} Is the average-case time complexity of the word problem in any finitely generated group of matrices over $\Q$ linear?
\medskip

\item[4.] The (worst-case) complexity of the word problem in a finite group is obviously linear,
but this may not be the case for the identity problem. We ask: given a finite group $G$, is there a polynomial-time algorithm for solving the identity problem in $G$?
\medskip

\item[5.] {\bf (a)} Is the average-case time complexity of the word problem in a free
Burnside group $B(m, d)$ of a sufficiently large odd exponent $d$ linear? It is known that the word problem in $B(m, d)$ is in the class {\bf NP}, see \cite[28.2]{Olshanskii}.
\medskip

\noindent {\bf (b)} Is the average-case time complexity of the identity problem in a Burnside variety $\mathcal{B}_d$ of groups of a sufficiently large odd exponent $d$ linear?
\medskip

\item[6.] Let $G$ be a finitely generated (or even a finitely presented) group where the average-case time
complexity of the word problem is linear. Is it necessarily the case that the generic-case time complexity of the word problem in $G$ is linear?

We note 
that if the sums in the formula (\ref{avcase1}) are bounded by a linear function of $n$, then for any superlinear function $f(n)$ the densities of subsets of $G$ where $T(w)>f(n)$ (for $w\in W_n$) are approaching 0 when $n$ goes to infinity. This means that if the average-case time complexity of the word problem in $G$ is linear, then the generic-case time complexity of the word problem in $G$ can be bounded by any superlinear function.

A similar argument shows that if the average-case time complexity of the word problem in $G$ is polynomial, then the generic-case time complexity of the word problem in $G$ is polynomial, too, although the degree of a polynomial may be a little higher. More accurately, if the average-case complexity is $O(n^d)$, then the generic-case complexity is $O(n^{d+\varepsilon})$ for any $\varepsilon > 0$.

Another implication that we can mention in this context is that the exponential time average-case time complexity of the word problem in $G$ implies the exponential time (as opposed to super-exponential time) worst-case complexity of the word problem in $G$. This is because super-exponential complexity of the word problem even on just a single input will dominate everything else, so the average-case complexity would be super-exponential as well.

\end{itemize}

%


\begin{thebibliography}{ABC}


\bibitem{Bass}
H. Bass, {\it The degree of polynomial growth of finitely generated groups}, Proc. London Math. Soc. {\bf 25} (1972), 603--612.

\bibitem{BMS}
G. Baumslag, A. G. Myasnikov, V. Shpilrain,
{\it Open problems  in combinatorial group theory. Second edition},
Contemp. Math., Amer. Math. Soc. {\bf 296} (2002), 1--38. \emph{}


\bibitem{Baumslag}
G. Baumslag, D. Solitar, {\it Some two-generator one-relator non-Hopfian groups}, Bull. Amer. Math. Soc. {\bf 68} (1962), 199--201.

\bibitem{Bigelow}
S. Bigelow, {\it Braid groups are linear}, J. Amer. Math. Soc. {\bf 14} (2001), 471--486.


\bibitem{Birget}
J.-C. Birget, A. Yu. Olshanskii, E. Rips, M. V. Sapir, {\it Isoperimetric functions of groups and computational complexity of the word problem}, Ann. of Math. (2) {\bf 156} (2002), 467--518.

\bibitem{BMS}
A. Borovik, A. G. Myasnikov, V. Shpilrain, {\it Measuring
sets in infinite groups}, Contemp. Math., Amer. Math. Soc. {\bf 298} (2002),
21--42.



\bibitem{CFP}
J. W. Cannon, W. J. Floyd, and W. R. Parry, {\it Introductory notes on Richard
Thompson's groups}, L'Enseignement Mathematique (2) {\bf 42} (1996),
215--256.

\bibitem{Baumslag}
A. E. Clement, S. Majewicz, and M. Zyman, {\sl The theory of nilpotent groups},
Birkh\"auser/Springer, Cham, 2017.

\bibitem{Cohen}
J. M. Cohen, {\it Cogrowth and amenability of discrete groups},  J. Funct. Anal. {\bf 48} (1982), 301--309.



\bibitem{Grigorchuk}
R. I. Grigorchuk, {\it Symmetrical random walks on discrete groups},
in: Advances in Probability and Related Topics, Vol. {\bf 6}, pp. 285--325, Marcel Dekker 1980.


\bibitem{Gromov}
M. Gromov, {\it Groups of polynomial growth and expanding maps},  Inst. Hautes {\'E}tudes Sci. Publ. Math.  {\bf 53} (1981), 53--73.


\bibitem{Groves}
J. R. J. Groves, {\it Varieties of soluble groups and a dichotomy of P. Hall},  Bull. Austral. Math. Soc. {\bf 5} (1971), 391--410.

\bibitem{Harvey}
D. Harvey and J. van der Hoeven, {\it Integer multiplication in time $O(n \log n)$}, Ann. of Math. {\bf 193} (2021), 563--617.

\bibitem{Hirsh}
K. A. Hirsch, {\it On infinite soluble groups. II}, Proc. London Math. Soc. (2) {\bf 44} (1938), 336--344.

\bibitem{KMSS}
I. Kapovich, A. G. Myasnikov, P. Schupp, V. Shpilrain,
{\it  Generic-case complexity, decision problems in group theory and random
walks}, J.  Algebra {\bf 264} (2003), 665--694.

\bibitem{KMSS2}
I. Kapovich, A. G. Myasnikov, P. Schupp, V. Shpilrain,
{\it  Average-case complexity and decision problems in group theory},
 Advances in Math. {\bf 190} (2005), 343--359.

\bibitem{KWM}
I. Kapovich, R. Weidmann, A. G. Myasnikov, {\it Foldings, graphs of groups and the membership problem}, Internat. J. Algebra Comput. {\bf 15} (2005), 95--128.

\bibitem{Karatsuba}
A. A. Karatsuba, {\it The complexity of computations.} (Russian) Trudy Mat. Inst. Steklov. {\bf 211}  (1995), 186--202.

\bibitem{OK}
O. G.  Kharlampovich, {\it A finitely presented solvable group with unsolvable word problem.} (Russian) Izv. Akad. Nauk SSSR Ser. Mat. {\bf  45} (1981), 852--873.

\bibitem{KMS}
O. G.  Kharlampovich, A. G. Myasnikov, M. Sapir, {\it Algorithmically complex residually finite groups}, Bull. Math. Sci. {\bf 7} (2017),  309--352.

\bibitem{Kleiman}
Yu. G. Kleiman, {\it On identities in groups}. (Russian)  Trudy Moskov. Mat. Obshch. {\bf 44} (1982), 62--108.

\bibitem{Knuth}
D. E. Knuth, {\it The analysis of algorithms}. Actes du Congr\`es International des Math\'ematiciens (Nice, 1970), Tome 3, pp. 269–-274. Gauthier-Villars, Paris, 1971.

\bibitem{Krammer}
D. Krammer, {\it Braid groups are linear}, Ann. of Math. (2) {\bf 155} (2002),  131--156.


\bibitem{Levin}
L. Levin, {\it Average case complete problems}, SIAM J. Comput. {\bf 15} (1986),  285--286.

\bibitem{LS}
R.~Lyndon and P.~Schupp, \emph{Combinatorial Group Theory,} Springer-Verlag, 1977. Reprinted in the ``Classics in  mathematics'' series, 2000.


\bibitem{MMNV}
J. Macdonald, A. Myasnikov, A. Nikolaev, S. Vassileva, {\it Logspace and compressed-word computations in nilpotent groups}, Trans. Amer. Math. Soc. {\bf 375} (2022), 5425--5459.

\bibitem{Master}
Master theorem, \url{https://en.wikipedia.org/wiki/Master_theorem_(analysis_of_algorithms)}

\bibitem{MU}
A. G. Miasnikov, A. Ushakov {\it Random van Kampen diagrams and algorithmic
problems in groups},  Groups Complex. Cryptol. {\bf 3} (2011),
121--185.

\bibitem{Vershik}
A. G. Myasnikov, V. Roman’kov, A. Ushakov, and A. Vershik, {\it The word and geodesic problems
in free solvable groups}, Trans. Amer. Math. Soc. {\bf 362} (2010), 4655--4682.

\bibitem{Neumann}
H. Neumann, {\sl Varieties of groups}, Springer, 1967.

\bibitem{Olshanskii}
A. Yu. Olshanskii, {\sl Geometry of Defining Relations in Groups}, Springer 1991.


\bibitem{Sale}
A. Sale, {\it Geometry of the conjugacy problem in lamplighter groups},  Algebra and computer science, Contemp. Math. {\bf 677}, Amer. Math. Soc. (2016), 171--183.

\bibitem{Strassen}
A. Sch\"onhage, V. Strassen, {\it Schnelle Multiplikation grosser Zahlen} (German), Computing (Arch. Elektron. Rechnen) {\bf 7} (1971), 281--292.

\bibitem{Sharp}
R. Sharp, {\it Local limit theorems for free groups}, Math. Ann. {\bf 321} (2001), 889--904.

\bibitem{sublinear}
V. Shpilrain, {\it Sublinear time algorithms in
the theory of groups and semigroups}, Illinois J. Math. {\bf 54}
(2011), 187--197.

\bibitem{SU}
V. Shpilrain and A. Ushakov, {\it Thompson's group and public key cryptography}, in  ACNS 2005, Lecture Notes
Comp. Sc. {\bf 3531} (2005), 151--164.


\bibitem{Swan}
R. Swan, {\it Representations of polycyclic groups},  Proc. Amer. Math. Soc. {\bf 18} (1967), 573--574.

\bibitem{Timoshenko}
E. I. Timoshenko, {\it Certain algorithmic questions for metabelian groups}, Algebra and Logic {\bf 12} (1973), 132--137.

\bibitem{Ushakov}
A. Ushakov, {\it Algorithmic theory of free solvable groups: randomized computations},  J. Algebra {\bf 407} (2014), 178--200.

\bibitem{Wehrfritz}
B. A. F. Wehrfritz, {\it On finitely generated soluble linear groups}, Math. Z. {\bf 170} (1980),  155--167.

\bibitem{Woess}
W. Woess, {\sl Random walks on infinite graphs and groups}, Cambridge Tracts in Mathematics, {\bf 138}. Cambridge University Press, Cambridge, 2000.




\end{thebibliography}
\end{document}